\documentclass[a4paper,10pt]{article}

\usepackage{amsmath, amstext, amsopn, amsthm, amscd, amsxtra, amsfonts, amssymb}
\usepackage[latin1]{inputenc}
\usepackage[german]{babel}
\usepackage{paralist}
\usepackage{hhline}
\usepackage{array}
\usepackage{begriff}

\newtheoremstyle{KURSIV}{\topsep}{\topsep}{\it}{}{\bf}{.}{5pt}{\thmnumber{#2}{.~}\thmname{#1}\thmnote{ #3}}
\theoremstyle{KURSIV}

\newtheorem{thm}{Theorem}
\newtheorem{lem}[thm]{Lemma}

\newtheoremstyle{PLAIN}{\topsep}{\topsep}{}{}{\bf}{.}{5pt}{\thmnumber{#2}{.~}\thmname{#1}\thmnote{ #3}}
\theoremstyle{PLAIN}

\newtheorem{dfn}[thm]{Definition}
\newtheorem{conv}[thm]{Konvention}
\newtheorem{ex}[thm]{Beispiel}
\newtheorem{rem}[thm]{Bemerkung}

\newtheorem{aufg}[thm]{\"Ubung}

\begin{document}

\title{Eine kurze Einf\"uhrung in \\Gottlob Freges \textquotedblleft{}Begriffsschrift\textquotedblright{}$^{\text{\hspace{0.4pt}1}}$}
\author{Sven-Ake Wegner$^{\text{\hspace{0.5pt}2}}$}
\date{6.~Mai 2008}
\maketitle

\renewcommand{\thefootnote}{}

\hspace{-1000pt}\footnote{$^{\text{1}}$Die folgenden Notizen erg\"anzen einen Vortrag des Autors im Seminar \textquotedblleft{}Geschichte der modernen\linebreak{}\phantom{a}\hspace{15.2pt}Logik\textquotedblright{} von Volker Peckhaus, gehalten an der Universit\"at Paderborn im Sommersemester 2008.\vspace{2pt}}

\hspace{-1000pt}\footnote{$^{\text{2}}$Mathematisches Institut, Universit\"at Paderborn, Warburger Stra\ss{}e 100, 33098 Paderborn.}

\vspace{-40pt}

\begin{abstract}
In seiner 1879 publizierten Arbeit \cite{Frege} \"uber die Begriffsschrift f\"uhrt Gottlob Frege eine Notation ein, um u.a. mathematische Argumentationen mithilfe derselben zu formalisieren. Im Folgenden erl\"autern wir die Fregesche Notation unter Verwendung der \"ublichen aussagenlogischen Bezeichnungen und diskutieren Unterschiede zu und Gemeinsamkeiten mit modernen Schreibweisen. Ferner res\"umieren wir einige der Methoden, die Frege in \cite{Frege} verwendet.
\end{abstract}


\section{Vorbemerkungen}\label{Introduction}

In der unter Mathematikern---nicht Logikern---gebr\"auchlichen Notation der Aussagenlogik verwendet man oft ein Gemisch aus Symbolen und Prosa. Zum Beispiel schreibt man h\"aufig Dinge, wie \textquotedblleft{}Sei $A$ eine Aussage.\textquotedblright{} und dann vielleicht nach einer Argumentation als Schlussfolgerung \textquotedblleft{}$A$ ist wahr.\textquotedblright{}. Oder aber man beginnt einen Beweis mit dem Satz \textquotedblleft{}Angenommen, $A$ ist nicht wahr.\textquotedblleft{} und folgert daraus einen Widerspruch. Obwohl obige S\"atze wesentliche Teile einer formalen Argumentation sind, gibt es dennoch keine Symbole, die anstelle von Prosa benutzt werden. Oft schreibt man auch die Worte \textquotedblleft{}f\"ur alle\textquotedblright{}, \textquotedblleft{}es existiert\textquotedblright{} und \textquotedblleft{}daraus folgt\textquotedblright{} aus, anstatt die in diesem Fall verf\"ugbaren Symbole \textquotedblleft{}$\forall$\textquotedblright{}, \textquotedblleft{}$\exists$\textquotedblright{}  und 
\textquotedblleft{}$\Rightarrow$\textquotedblright{} zu benutzen.
\smallskip
\\ Wieviele und welche Teile einer mathematischen Argumentation in Prosa bzw. mithilfe von Symbolen formuliert werden, ist stark unterschiedlich und meist eine Stil-- oder Geschmacksfrage. In Gottlob Freges Begriffsschrift gibt es kein derartiges Gemisch. In seiner Einleitung formuliert er es als sein Ziel einen Kalk\"ul zu entwickeln, der nach strengen formalen Regeln funktioniert und Argumentationen erlaubt, die---im Gegensatz zu Prosa---keine unterschiedlichen Interpretationen zulassen und f\"uhrt in konsequenter Weise eine Notation ein, die v\"ollig ohne Prosa auskommt. Ferner verzichtet er z.B.  auf ein eigenes Symbol f\"ur den Existenzquantor und dr\"uckt diesen nur durch sein Verneinungssymbol und den Allquantor aus. Nat\"urlich k\"onnte man dies auch in der modernen Notation machen (und h\"atte ein Symbol weniger, welches man kennen mu\ss{}), jedoch w\"urde dies der Lesbarkeit einer Argumentation nicht 
gerade zutr\"aglich sein. 
\smallskip
\\Im Folgenden stellen wir die Fregesche Notation vor und bedienen uns dabei der Sprache und Symbolik, wie sie in der Mathematik \"ublicherweise benutzt wird. Da in der Begriffsschrift ein aussagenlogischer Ausdruck ein aus mehreren Einzelteilen zusammengesetztes zweidimensionales Symbol ist und im Gegensatz zur modernen Notation nicht einfach als Zeile von links nach rechts gelesen werden kann, m\"ussen wir in den folgenden Definitionen stets ganze Aussagen betrachten. Die Begriffsschrift ist also nicht einfach nur ein anderer \textquotedblleft{}Zeichensatz\textquotedblright{} als die heute \"ubliche Notation.

\section{Die Begriffsschrift}

Wir verschieben die wesentliche Diskussion dar\"uber, was im Fregeschen Sinne Aussagen, Behauptungen, etc. sind auf Abschnitt \ref{3} und verwenden zun\"achst die Termini \textquotedblleft{}Aussage\textquotedblright{}, \textquotedblleft{}aussagenlogischer Ausdruck\textquotedblright{} und \textquotedblleft{}Aussagevariable\textquotedblright{} im \"ublichen (naiven) Sinne, d.h. eine Aussage ist ein sinnvolles sprachliches Gebilde, welchem sich ein \textquotedblleft{}Wahrheitswert\textquotedblright{} ($w$ oder $f$) zuordnen l\"a\ss{}t. Sinnvolle sprachliche Gebilde, die nach Belegung gewisser Eingangsgr\"o\ss{}en einen Wahrheitswert erhalten nennen wir aussagenlogische Ausdr\"ucke und die gewissen Eingangsgr\"o\ss{}en Aussagevariablen. Es ist uns bewu\ss{}t, da\ss{} wir soeben f\"ur keines der drei Worte eine formale Definition gegeben haben und betrachten Obiges daher als eine Umschreibung, die---auf viele gutartige Beispiele angewand---intuitiv verst\"andlich ist. In diesem 
Sinne sind auch die folgenden Definitionen zu verstehen.
\smallskip
\\Wir werden im Folgenden Abk\"urzungen f\"ur Aussagen und aussagenlogische Ausdr\"ucke verwenden, bzw. Platzhalter oder auch Variablen, die eine beliebige Aussage oder einen beliebigen aussagenlogischen Ausdruck repr\"asentieren.

\begin{conv} Wie schon im ersten Abschnitt diskutiert, gibt es in der heute \"ublichen Notation keine Symbole, um zum Ausdruck zu bringen, da\ss{} eine Variable eine Aussage repr\"asentieren soll, bzw. da\ss{} diese Aussage wahr ist. In der Fregeschen Notation hingegen gibt es hierf\"ur Symbole und diese sind in die weitere Symbolik \textquotedblleft{}integriert\textquotedblright{} in einem Sinne den die folgenden Definitionen zeigen werden. Wir halten zun\"achst fest, da\ss{} wir anstatt zu Beginn eines Abschnittes an der Stelle wo wir unsere Bezeichnungen festlegen zu bemerken, da\ss{} ein gewisser Buchstabe (z.B. $A$) im Folgenden f\"ur eine Aussage stehe, \"uberhaupt nichts derartiges bemerken, sondern ab dem ersten Auftreten von $A$,
$$
\BGcontent{_A}
$$
schreiben. Wollen wir zum Ausdruck bringen, da\ss{} die Aussage $\textbf{--}\,A$ wahr ist, so verzichten wir auf die Worte \textquotedblleft{}Es gilt\textquotedblright{} unmittelbar vor unserer Aussage (oder unserem aussagenlogischen Ausdruck), und schreiben stattdessen
$$
\BGassert{_A}.
$$
\end{conv}

Im Folgenden werden wir h\"aufig zwischen der heute \"ublichen Notation und der Begriffsschrift wechseln. Dabei ist dann $A$ (ohne den waagerechten Strich) eine Aussage in heutiger Notation und $\textbf{--}\,A$ dieselbe Aussage in Begriffsschrift. Um die folgenden Erkl\"arungen nicht unn\"otig zu komplizieren, verzichten wir im Folgenden darauf, jedesmal auf den Unterschied hinzuweisen und \"uberlassen dies dem Leser.

\begin{dfn}
F\"ur die Verneinung $\neg A$ unserer Aussage $A$ schreiben wir in Begriffsschrift
$$
\BGnot{_A}.
$$
\end{dfn}

Wir wissen, da\ss{} von den Junktoren $\neg$, $\vee$, $\wedge$, $\Rightarrow$, etc. die Verneinung und ein beliebiger anderer gen\"ugen, um damit alle weiteren auszudr\"ucken, wobei es jeweils unterschiedliche M\"oglichkeiten gibt, dies zu tun. Wie schon in Abschnitt 1 erw\"ahnt, ist die Begriffsschrift redundanzfrei. Es gen\"ugt also ein weiteres Zeichen.

\begin{dfn} Seien $\textbf{--}\,A$ und $\textbf{--}\,B$ zwei Aussagen. Dann bezeichnet
$$
\setlength{\BGlinewidth}{40pt}\BGconditional{\BGterm{_B}}{\BGterm{_A}}
$$
die Aussage $B\Rightarrow A$. Obiges Zeichen entspricht also dem heute \"ublichen Implikationspfeil.
\end{dfn}

\begin{rem} Im Folgenden werden wir es oft mit zusammengesetzten Zeichen zu tun haben. Dabei denken wir uns vor jedem senkrechten \textquotedblleft{}Bedingungsstrich\textquotedblright{} eine Klammer, die am Ende des entsprechenden Teilzeichens geschlossen wird. Zum Beispiel bezieht sich die Verneinung in \ref{lem1}.(ii) auf alles rechts davon Stehende --- was im Beweis auch nochmal verdeutlicht wird.
\end{rem}

\begin{lem}\label{lem1}\begin{compactitem}\item[(i)]In Begriffsschrift entspricht die Disjunktion von $\textbf{--}\,A$ und $\textbf{--}\,B$ dem folgenden Zeichen.
$$
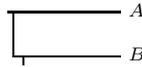
\setlength{\BGlinewidth}{40pt}\BGconditional{\BGnot\BGterm{_B}}{\BGterm{_A}}
$$
\item[(ii)] Analog entspricht die Konjunktion dem Folgenden.
$$
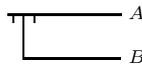
\setlength{\BGlinewidth}{40pt}\BGnot\BGconditional{\BGterm{_B}}{\BGnot \BGterm{_A}}
$$
\end{compactitem}
\end{lem}
\begin{proof} (i) Das Begriffsschrift--Symbol \"ubersetzt lautet $\neg B\Rightarrow A$. Per Wahrheitstafel
$$\begin{array}{|c|c|c|c|c|c|}\hline
A & B & \neg B &\neg B\Rightarrow A & A\vee B \\\hline
w & w & f & w & w \\\hline
w & f & w & w & w \\\hline
f & w & f & w & w \\\hline
f & f & w & f & f \\\hline
\end{array}$$
folgt die \"Aquivalenz zu $A\vee B$.
\medskip
\\(ii) Wir haben
$$
\big[\neg (B\Rightarrow \neg A)\big] \Leftrightarrow \big[A\wedge B\big],
$$
wegen der folgenden Wahrheitstafel.
$$\begin{array}{|c|c|c|c|c|c|}\hline
A & B & \neg A & B\Rightarrow \neg A &\neg(B\Rightarrow \neg A) & A\wedge B \\\hline
w & w & f & f & w &  w\\\hline
w & f & f & w & f &  f\\\hline
f & w & w & w & f &  f\\\hline
f & f & w & w & f &  f\\\hline
\end{array}$$
\end{proof}

\begin{ex} Die Aussage
$$
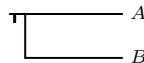
\setlength{\BGlinewidth}{40pt}\BGnot\BGconditional{\BGterm{_B}}{\BGterm{_A}}
$$
lautet in moderner Notation
$$
B\wedge \neg A.
$$
Wie oben sieht man dies per \"Ubersetzung und mittels geeigneter Wahrheitstafeln.
\end{ex}

In Abschnitt 1 haben wir auch schon erw\"ahnt, da\ss{} von den Symbolen $\forall$ und $\exists$ eines gen\"ugen w\"urde, da sich das jeweils andere dadurch ausdr\"ucken l\"a\ss{}t. Wir definieren nur ein Symbol f\"ur den Allquantor und folgern dann, welches Zeichen f\"ur den Existenzquantor steht.

\begin{dfn} Sei $\textbf{--}\,F(\cdot)$ ein aussagenlogischer Ausdruck und $x$ eine Aussagenvariable. Dann bezeichnet
$$
\setlength{\BGlinewidth}{50pt}
\BGcontent\BGquant{\text{$x$}}\BGterm{_{F(x)}}
$$
die Aussage $\forall\:x\colon F(x)$.
\end{dfn}

\begin{lem} Das Symbol
$$
\setlength{\BGlinewidth}{50pt}
\BGnot\BGquant{\text{$x$}}\BGnot\BGterm{_{F(x)}}
$$
entspricht in moderner Schreibweise der Aussage $\exists \:x\,\colon F(x)$.
\end{lem}
\begin{proof}Per Definition entspricht obiges gerade $\neg(\forall\:x\,\colon \neg F(x))$. Mit den \"ublichen Regeln \"uber den Umgang mit Quantoren sieht man sofort, da\ss{} dies \"aquivalent zu $\exists\:x\,\colon F(x)$ ist.
\end{proof}

\begin{ex} Die Begriffsschrift l\"a\ss{}t nun in nat\"urlicher Weise aussagenlogische Ausdr\"ucke in denen mehr als eine Variable vorkommt zu: Sind $F(\cdot)$ und $R(\cdot,\cdot)$ aussagenlogische Ausdr\"ucke und $x$, $y$ Aussagenvariablen, so lautet das Begriffsschrift--Symbol f\"ur die Feststellung
\begin{center}
\textquotedblleft{}Es gilt: $\forall\:x \text{ mit }\,F(x)\;\exists\:y\colon R(x,y)$\textquotedblright{}
\end{center}
gerade
$$
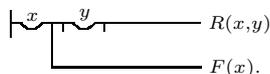
\setlength{\BGlinewidth}{70pt}
\BGassert\BGquant{\text{$x$}}\BGconditional{\BGterm{_{F(x).}}}{\BGnot\BGquant{\text{$y$}}\BGnot \BGterm{_{R(x,y)}}}
$$
\end{ex}

Iterationen sind hierbei keine Grenzen gesetzt, wie die folgenden \"Ubungsaufgaben zeigen. Ferner lassen sich auch formale Operationen (wie z.B. die Negation einer Aussage zu formulieren) in der Begriffsschrift \"ahnlich leicht bewerkstellingen, wie in der modernen Quantorenschreibweise.

\begin{aufg} Schreibe den folgenden aussagenlogischen Ausdruck in moderner Notation. Beachte: $f\colon\mathbb{R}\rightarrow \mathbb{R}$ bezeichnet eine Abbildung und $x_0\in\mathbb{R}$ ist beliebig.
$$
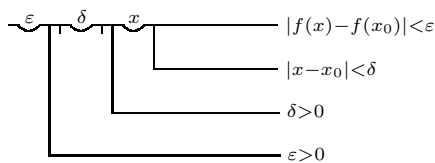
\setlength{\BGlinewidth}{100pt}
\BGcontent\BGquant{\varepsilon}\BGconditional{\BGterm{_{\varepsilon>0}}}{\BGnot\BGquant{\delta}\BGnot \BGconditional{\BGterm{_{\delta>0}}}{\BGquant{\text{$x$}} \BGconditional{\BGterm{_{|x-x_0|<\delta}}}{   \BGterm{_{|f(x)-f(x_0)|<\varepsilon}}}}}
$$
\end{aufg}

\textit{L\"osung:} Hier steht die Definition daf\"ur da\ss{} $f$ stetig an der Stelle $x_0$ ist, n\"amlich
$$
\forall\:\varepsilon>0\;\exists\:\delta>0\;\forall\:x,\,|x-x_0|<\delta\, \colon |f(x)-f(x_0)|<\varepsilon.
$$

\begin{aufg} Schreibe die Aussage
$$
\forall\:n\in\mathbb{N}\;\exists \: m\geqslant n\,,k\:\forall\:\mu\geqslant m,\,l,\,\varepsilon>0\;\exists\:L,\,S>0\,\colon v_{m,l}\,\leqslant\,\max(\varepsilon v_{n,k},\,S v_{\mu,L})
$$
mithilfe von Freges Begriffsschrift und bestimme ihre Verneinung.
\end{aufg}

\textit{L\"osung:} Obige Bedingung sieht in Freges Notation folgenderma\ss{}en aus:
$$
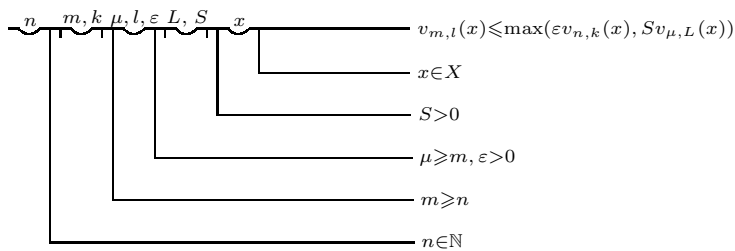
\setlength{\BGlinewidth}{150pt}
\BGcontent\BGquant{\text{$n$}}\BGconditional{\BGterm{_{n\in\mathbb{N}}}}{\BGnot\BGquant{\text{$m, k$}}\BGnot \BGconditional{\BGterm{_{m\geqslant n}}}{\BGquant{\text{$\mu,l,\varepsilon$}} \BGconditional{\BGterm{_{\mu\geqslant m,\, \varepsilon>0}}}{\BGnot \BGquant{\text{$L,\,S$}}\BGnot \BGconditional{\BGterm{_{S>0}}}{    \BGquant{\text{$x$}}\BGconditional{\BGterm{_{x\in X}}}{   \BGterm{_{v_{m,l}(x)\leqslant\max(\varepsilon v_{n,k}(x),\,S v_{\mu,L}(x))}}}}}}}
$$
Ihre Verneinung gewinnt man durch \textquotedblleft{}umdrehen\textquotedblright{} der Quantoren, sowie der Absch\"atzung:
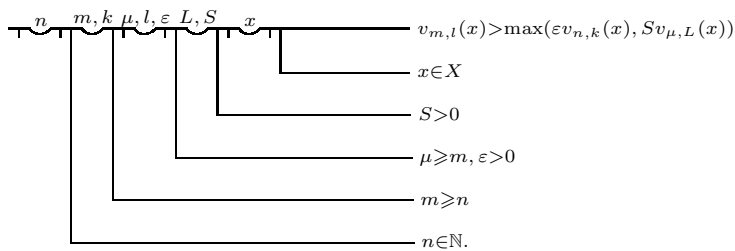
$$
\setlength{\BGlinewidth}{150pt}
\BGcontent\BGnot\BGquant{\text{$n$}}\BGnot\BGconditional{\BGterm{_{n\in\mathbb{N}.}}}{\BGquant{\text{$m,k$}} \BGconditional{\BGterm{_{m\geqslant n}}}{\BGnot\BGquant{\text{$\mu,l,\varepsilon$}}\BGnot \BGconditional{\BGterm{_{\mu\geqslant m,\, \varepsilon>0}}}{ \BGquant{\text{$L,S$}} \BGconditional{\BGterm{_{S>0}}}{  \BGnot\BGquant{\text{$x$}}\BGnot \BGconditional{\BGterm{_{x\in X}}}{   \BGterm{_{v_{m,l}(x)>\max(\varepsilon v_{n,k}(x),\,S v_{\mu,L}(x))}}}}}}}
$$

\section{Diskussion}\label{3}

Aus heutiger Sicht mag die Fregesche Notation umst\"andlich erscheinen. Jedoch mu\ss{} ber\"ucksichtig werden, da\ss{} Freges Konzept in nat\"urlicher und selbstverst\"andlicher Weise all diejenigen Hilfsmittel zur Verf\"ugung stellt, die auch heute noch am Anfang eines jeden Mathematikstudiums stehen.
\smallskip
\\Was seine Pr\"asentation in \cite{Frege} angeht, so ist diese ebenfalls durchaus mit einer modernen (nicht f\"ur Logiker aber wohl f\"ur Mathematiker) vergleichbar: Nach einer Einf\"uhrung, in welcher er darauf hinweist, da\ss{} die Sprache ein zu unpr\"azises Mittel f\"ur die formale Begr\"undung von Sachverhalten ist (und den Unterschied zwischen Beweis und Heuristik thematisiert), f\"uhrt er die Begriffe \textquotedblleft{}Inhalt\textquotedblright{} und \textquotedblleft{}Urteil\textquotedblright{} ein; heute \"ublich sind die Bezeichnungen \textquotedblleft{}Aussage\textquotedblright{} und vielleicht \textquotedblleft{}Aussage mit Wahrheitswert $w$\textquotedblright{}. In seinen Bemerkungen zu den Begriffen \textquotedblleft{}Subjekt\textquotedblright{} und \textquotedblleft{}Pr\"adikat\textquotedblright{} verdeutlicht er (nochmals), da\ss{} er mit seiner Notation grammatikalische Probleme zu \"uberwinden versucht. Der von ihm eingef\"uhrte \textquotedblleft{}Bedingungsstrich\textquotedblright{} (die 
Implikation) wird 
de facto mithilfe von Wahrheitstafeln definiert --- auch wenn er an dieser und den folgenden Stellen stets die einzelnen Zeilen in Prosa durchgeht. Er stellt fest, da\ss{} nur die Implikation n\"otig ist, um aus bekannten Schlu\ss{}weisen neue zu gewinnen. Seine Begriffe \textquotedblleft{}Function\textquotedblright{} und \textquotedblleft{}Argument\textquotedblright{} entsprechen gerade dem was heute unter dem Namen \textquotedblleft{}aussagenlogischer Ausdruck\textquotedblright{} (oder auch \textquotedblleft{}Aussageform\textquotedblright{}) und \textquotedblleft{}Aussagenvariable\textquotedblright{} auf analoge Weise eingef\"uhrt wird. Nach der Definition des Existenzquantors findet sich sogar explizit die Erkl\"arung \textquotedblleft{}Es gibt...\textquotedblright{}.
\smallskip
\\Zusammenfassend l\"a\ss{}t sich das erste Kapitel von \cite{Frege} als eine Einf\"uhrung in die Aussagenlogik betrachten, die sich lediglich in grammatikalischen und symbolischen Belangen von modernen Darstellungen unterscheidet.
\smallskip
\\F\"ur Informationen zur Person und anderen Arbeiten Gottlob Freges siehe z.B. Dummett \cite{Dummett, Dummett2}, George und Heck \cite{Heck} sowie Sternfeld \cite{Sternfeld} und die dort angegebenen Referenzen.

\end{document}